\documentclass{amsart}
\usepackage{amsmath,amssymb,amsthm}
\usepackage[dvipdfm]{graphicx}
\usepackage{enumerate}
\usepackage[dvips]{color}
\usepackage{version}

\newtheorem{Thm}{Theorem}[section]
\newtheorem{Lem}[Thm]{Lemma}

\newtheorem{Cor}[Thm]{Corollary}
\newtheorem{Prop}[Thm]{Proposition}

\newtheorem{``Conj"}[Thm]{``Conjecture"}

\theoremstyle{remark}

\theoremstyle{definition}
\newtheorem{Def}[Thm]{Definition}

\newtheorem*{ack}{Acknowledgments}

\newcommand{\Proj}{\mathop{\mathrm{Proj}}\nolimits}

\newcommand{\DF}{\mathop{\mathrm{DF}}\nolimits}

\begin{document}

\title[Remarks on test configurations]
{On Parametrization, optimization and triviality of test configurations}

\author{Yuji Odaka}
\address{Research Institute for Mathematical Sciences, Kyoto University,
Kyoto, Japan}
\email{yodaka@kurims.kyoto-u.ac.jp}

%\date{19th, February, 2012}

\maketitle

%%%%%%%%%%%%%%%%%%%%%%%%%%
\begin{abstract}
We give a parametrization of test configurations 
in the sense of Donaldson via spherical buildings, 
and show the existence of ``optimal" destabilizing test configurations 
for unstable varieties, in the wake of Mumford and Kempf. 
We also give an account of the recent slight amendment to definition of K-stability after 
Li-Xu, from two other viewpoints: 
from the one parameter subgroups and 
from the author's blow up formalism. 
\end{abstract}
%%%%%%%%%%%%%%%%%%%%%%%%%%%
\tableofcontents

\section{Introduction}

The concept of test configurations of polarized varieties is introduced by Donaldson \cite{Don02} to generalize K-stability after Tian \cite{Tia97} which is a certain version of GIT stability \cite{Mum65} for polarized varieties. 
Indeed, to each test configuration, certain GIT weights correspond, 
which is a motivation of its introduction. 
Precisely speaking, 
they can be regarded as certain {\it equivalent classes} of 
one parameter subgroups of the general linear group ${\rm GL}(H^0(X,L^{\otimes r}))$ 
for a given polarized variety $(X,L)$. 

The aim of this short notes is twofold. First is to point out that, roughly speaking, 
dividing sets of test configurations further, all equivalent classes can be parametrized by a certain simplicial complex, the so-called {\it spherical building} invented by J.~Tits (cf.\ \cite{Tit74}). Moreover, thanks to a theorem of Rousseau \cite{Rou78} and Kempf \cite{Kem78}, originally known as the Tit's center conjecture, we conclude the existence of ``optimal" destabilizing test configurations. 
The outcome is a geometric analogue of the Harder-Narasimhan filtrations of unstable sheaves on a fixed  variety.  (The analogy becomes clearer after recent work of \cite{GSZ11}. Namely, it claims that the filtration coincides the optimal destabilizing objects in the sense of \cite{Rou78}, 
\cite{Kem78}. ) We remark here that the spherical building has already been discussed in 
the Mumford's classic \cite{Mum65} 
to give applications to stability problem although he used the 
word ``flag complex" with completely same meaning. Thus, we should note that 
our statements are essentially  consequences of \cite{Mum65}, \cite{Rou78}, \cite{Kem78}, and only some small pieces of arguments, so that they are not essentially new. However, it seems to the author that the results have not been recognized by the mathematical community of this area,  which is one motivation for him of writing this notes. 

The second objective of this notes is to explain from our point of view, 
about the pathological (almost trivial but not trivial) 
test configurations whose existence is pointed out in \cite[subsection 2.2]{LX11}. 
They addressed that we need to slightly modify the definition of K-stability by 
regarding them as trivial or ``throwing them away". For details, 
see subsection \ref{patho.ss}. In that subsection, 
we also explain how the blow up formalism \cite{Od11} 
works for that amended notion of K-stability. Moreover, 
the sub-locus of the spherical building which parametrizes 
those test configurations of specific kind will be described as well. 

\section{Statements}

We work over an arbitrary algebraically closed field $k$. 
We use the terminologies of line bundles and invertible sheaves interchangeably. 
By a polarized variety $(X,L)$, we mean that $X$ is a projective equidimensional 
variety and $L$ is an ample line bundle on it. 
At the beginning, 
we recall the definition of (semi) test configurations and 
introduce certain (equivalence) relations for test configurations as follows. 

\begin{Def}\label{tc}
A \textit{test configuration} (resp.\ \textit{semi test configuration}) for a polarized variety $(X,L)$ is a quasi-projective variety $\mathcal{X}$ with 
an invertible sheaf $\mathcal{L}$ on it with: 
\begin{enumerate}
\item{a $\mathbb{G}_{m}$ action on $(\mathcal{X},\mathcal{L})$}
\item{a projective flat morphism $p\colon \mathcal{X} \rightarrow \mathbb{A}^{1}$}
\end{enumerate}
such that $p$ is $\mathbb{G}_{m}$-equivariant for the usual action on $\mathbb{A}^{1}$: 
\begin{align*}
\mathbb{G}_{m}\times \mathbb{A}^{1}&& \longrightarrow&& \mathbb{A}^{1}\\
                          (t,x)    && \longmapsto    &&    tx,      
\end{align*}
$\mathcal{L}$ is relatively ample (resp.\ relatively semi ample), 
and $(\mathcal{X},\mathcal{L})|_{p^{-1}(\mathbb{A}^{1}\setminus \{0\})}$ is $\mathbb{G}_{m}$-equivariantly isomorphic 
to $(X,L^{\otimes r})\times (\mathbb{A}^{1}\setminus \{0\})$ for some positive integer $r$, called \textit{exponent}, 
with the natural action of $\mathbb{G}_{m}$ on the latter and the trivial action on the former. 
\end{Def}

\begin{Def}\label{eq.tc}
For a given polarized variety $(X,L)$, we introduce the following three 
pre-equivalent conditions. We call two test configurations $(\mathcal{X}_i,
\mathcal{L}_i) (i=1,2)$

\begin{itemize}
\item {\it P-pre-equivalent} if $(\mathcal{X}_2,\mathcal{L}_2)$ is a base 
change of $(\mathcal{X}_1,\mathcal{L}_1)$ by the power morphism $t\mapsto t^m$ of $\mathbb{A}^1\rightarrow \mathbb{A}^1$ with some $\mathbb{Z}_{>0}$. 

\item {\it Q-pre-equivalent} if $(\mathcal{X}_2,\mathcal{L}_2)=(\mathcal{X}_1,\mathcal{L}_1^{\otimes l})$ equivariantly with $l \in \mathbb{Z}_{>0}$. 

\item {\it R-equivalent} when there is a $\mathbb{G}_m$-equivariant 
isomorphism $\phi \colon \mathcal{X}_1\cong \mathcal{X}_2$ over $\mathbb{A}^1$ 
with $\mathcal{L}_1\cong \phi^*\mathcal{L}_2$. To say in other (but a little rough)  words, it is exactly when they are isomorphic ``except for the linearizations". 

\end{itemize}

Consider an equivalent relation among test configurations which is 
generated by P-pre-equivariance relation and R-equivariance relation 
(resp.\ P-pre-equivariance relation, Q-pre-equivariance relation and R-equivalence 
relation) and call it T-equivalence relation (resp.\ U-equivalence relation). 

A T-equivalent class of test configuration $(\mathcal{X},\mathcal{L})$ will be called a {\it test degeneration} and we write it as $\overline{(\mathcal{X},\mathcal{L})}$. Note that, by the base change, central fiber does not change scheme-theoritically 
(including polarization), which is the origin of this name. 
Note that, exponents are well-defined for each test degeneration. 

A U-equivalent class of test configuration 
will be called a {\it test class}. Of course, a test degeneration gives a test class. 
\end{Def}

Let us assume the polarization $L$ attached to $X$ is {\it very} ample. 
Then, we have the following. 
We explain the rigorous meaning in subsections \ref{ss.bld}, \ref{pr.22}. 

\begin{Thm}\label{tc.bld}
Test degenerations $\overline{(\mathcal{X},\mathcal{L})}$ 
with very ample $\mathcal{L}$ of exponent $1$ 
are parametrized by the set of rational points 
$|\Delta({\rm GL}(H^0(X,L)))|_{\mathbb{Q}}$ of the 
spherical building $\Delta({\rm GL}(H^0(X,L))$. 
\end{Thm}

\noindent
Here, the spherical building $\Delta({\rm GL}(H^0(X,L))$ is 
a certain simplicial complex. Its geometric realization 
$|\Delta({\rm GL}(H^0(X,L))|$, which we donote as $\Delta_r$, with a natural topology is 
known to have a certain natural compatible metric structure 
(cf.\ \cite[Chapter 12]{AB08}). 
We will give the definition of the building in subsection \ref{ss.bld}. 
We will also see that although it is non-compact with respect to the usual topology, 
{\it it will be compact} with certain ``much stuck" topology to be introduced later. 
Without fixing exponents, we thus also have: 

\begin{Cor}\label{tc.bld.c}
Test classes of $(X,L)$ are parametrized by 
$\Delta_{\infty}:=\varinjlim |\Delta({\rm GL}(H^0(X,L^{\otimes r})))|_{\mathbb{Q}}$. 
\end{Cor}

\noindent
Here, the inductive system is given by taking power of polarization of test configurations 
i.e., $(\mathcal{X},\mathcal{L})\rightsquigarrow (\mathcal{X},\mathcal{L}^{\otimes l})$ 
with $l\in \mathbb{Z}_{>0}$. Note that these $\Delta_r$ and $\Delta_\infty$ 
can be expected to correspond to space of geodesics of (finite-dimensional or 
infinite-dimensional) symmetric spaces. 

For the basic of stability, we refer to \cite{Mum65}, \cite{Od11}. 
Based on the continuous (and piecewise-linear) behavior of GIT weights \cite[Prop 2.14]{Mum65}, we see the continuity of normalized Chow weights: 

\begin{Prop}\label{Cwt}
The normalized Chow weights form continous function on $|\Delta({\rm GL}(H^0(X,L^{\otimes r})))|_{\mathbb{Q}}$ with respect to the usual topology. 
On the other hand, it is a lower semi-continuous function with regards to 
the ``stuck" topology. 
\end{Prop}

\noindent
For the definitions of stuck topology and normalized Chow weights, see subsection \ref{ss.bld} and section \ref{s.wt} respectively. 
We also have the following abstract existence theorems 
of ``optimal" destabilizing objects, thanks to \cite{Rou78}, \cite{Kem78}. 

\begin{Cor}[{of Theorem \ref{tc.bld}}]\label{C.DO}
If $[X\subset \mathbb{P}]$ is Chow-unstable, it has a unique maximally destabilizing  
test degeneration. 
\end{Cor}

\begin{Cor}\label{K.DO}
If $(X,L)$ is K-unstable, it has a unique 
maximally destabilizing series of test degenerations. 
\end{Cor}

\noindent
For the definitions of these ``maximally destabilizing" objects, we will 
explain in section \ref{s.wt} but they are essentially with regard to the restricted 
action of ${\rm SL}(H^0(X,L^{\otimes r}))$. 
Note that, since if $(X,L)$ is K-unstable, 
$X\hookrightarrow \mathbb{P}(H^0(X,L^{\otimes r}))$
\footnote{In the present notes and other former papers of the author, 
the notation for the projectivization $\mathbb{P}(-)$ of 
vector space is the contravariant one, following Grothendieck. }
is Chow unstable for 
sufficiently large $r\in \mathbb{Z}_{>0}$ (cf.\ e.g. \cite[section 3]{RT07}), 
Corollary \ref{K.DO} simply follows from Corollary \ref{C.DO}. 
It would be interesting to know the behavior of Donaldson-Futaki invariants 
on $\Delta_{\infty}$ or to compare Corollary \ref{K.DO} with the frameworks 
of newer stability notions of \cite{Don10} and \cite{Sze11}. 

It would be also interesting to discuss relation with the similar existence statement in \cite[Theorem4]{Sze08} which was for toric case. He used a certain compactness property of the space of convex functions 
on the Delzant polytope (for which he refers to \cite[Proposition 5.2.6]{Don02}). 

\section{Parameter space}\label{sec.bld}

\subsection{Review of spherical building}\label{ss.bld}

See \cite{Tit74}, \cite{AB08} for more detailed accounts for the materials 
presented in this subsection. In this subsection, $G$ is assumed to be an arbitrary 
reductive algebraic group. 

\begin{Def}
The {\it spherical building} $\Delta(G)$ of $G$ is an abstract simplicial complex 
whose vertices represent maximal (proper) parabolic subgroups and faces represent 
finite set of maximal parabolic subgroups whose intersection is again parabolic. 
\end{Def}

Then, it is known that the following holds. 

\begin{Prop}[{cf.\ \cite[Chapter 2, section 2]{Mum65}}]\label{Mum.bld}
The rational points of spherical building of $G$ has natural correspondence with 
classes of one parameter subgroups as follows. We denote the center subgroup 
of $G$ as $Z$ and ``$1$-ps" stands for one parameter subgroup. 
$$|\Delta(G)|_{\mathbb{Q}}\cong\{ 1\mbox{-ps } 
\lambda\colon \mathbb{G}_m\rightarrow G 
\mbox{ whose image is not in } Z \}/\sim. $$
\noindent
Here, the equivalence relation is generated by the following three relations: 
\begin{enumerate}
\item $\lambda^m \sim \lambda^n$ for any $m,n\in \mathbb{Z}_{>0}$ 
\item $\lambda\sim p\lambda p^{-1}$ if 
$p\in (P(\lambda))_k$, where $P(\lambda)$ is a parabolic subgroup of $G$ 
whose $k$-valued points are: \\ 
$\{\gamma \in (G)_k \mid \lambda(t)\gamma\lambda(t^{-1}) 
\mbox{ has a specialization in } G \mbox{ when } t \mbox{ specializes to } 0 \}$. 
\item $\lambda \sim \lambda \xi$ where $\xi$ is a one parameter subgroup 
of $Z$. 
\end{enumerate}

\end{Prop}

Note that by the relation ${\rm (iii)}$, 
and not thinking one parameter subgroup with image in $Z$, 
we are essentially considering semisimple part. 
Of course, fixing a maximal torus $T$, we have a natural map 
$|\Delta(T)|\rightarrow|\Delta(G)|$ which is indeed injective. 
Thus, $|\Delta(G)|$ is a union of 
$S(T,\mathbb{Q}):=
({\rm Hom}_{\rm alg.grp}(\mathbb{G}_m,T)\otimes_{\mathbb{Z}}\mathbb{Q})
/\mathbb{Q}_{>0}$ 
which is dense in 
$S(T,\mathbb{R}):=
({\rm Hom}_{\rm alg.grp}(\mathbb{G}_m,T)\otimes_{\mathbb{Z}}\mathbb{R})
/\mathbb{R}_{>0}$ (apartments) 
naturally homeomorphic to sphere. 

It is not compact in the usual topology (see \cite[Chapter 12]{AB08} for 
compatible metric structure and its developments) but it is compact if we introduce a new topology as follows. 

Note that a variety $G/N(T)$ parametrizes set of maximal tori where $N(T)$ 
denotes the normalizer (algebraic) subgroup of a fixed maximal torus $T$ of $G$. 
Thus, we have a natural map 
$\phi\colon S(T,\mathbb{R})\times (G/N(T))_k \rightarrow |\Delta(G)|$. 
Using this, we can define our new topology as its associated quotient (compact) topology 
of the product of the Euclidean topology on $S(T,\mathbb{R})$ and the Zariski topology 
on $(G/N(T))_k$, via $\phi$. We call it {\it stuck topology} of $|\Delta(G)|$ (or 
$|\Delta(G)|_{\mathbb{Q}}$, defined by restriction) in this notes. 
We fix the above notation from now on. 

\subsection{Proof of Theorem \ref{tc.bld}}\label{pr.22}

Note that we are interested in the case $G={\rm GL}(H^0(X,L))$ 
which we assume in this subsection. 

As Proposition \ref{Mum.bld} shows, 
the rational points of the spherical building $|\Delta(G)|_{\mathbb{Q}}$ represent 
the classes of one parameter subgroups of ${\rm GL}(H^0(X,L))$ 
with respect to the equivalence relation generated by {\rm (i), (ii), (iii)} 
which are nontrivial.  

On the other hand, let us recall that each test configuration 
$(\mathcal{X},\mathcal{L})$ with very ample $\mathcal{L}$ of exponent $1$ 
and projection morphism 
$p\colon \mathcal{X}\rightarrow \mathbb{A}^1$ is 
isomorphic to the DeConcini-Procesi family of a one parameter subgroup 
(\cite[Proposition 3.7]{RT07}). 
More precisely, if we fix $\mathbb{G}_m$-equivariant trivialization 
$f\colon p_*\mathcal{L}\cong \mathcal{O}^{\oplus (M+1)}$ of 
$p_*\mathcal{L}$ then we can take a one parameter subgroup  $\lambda_f$ 
whose corresponding DeConcini-Procesi family is isomorphic to the original 
$(\mathcal{X},\mathcal{L})$. 

Let us fix a $\mathbb{G}_m$-equivariant trivialization $f$ 
and take associated eigenspace decomposition of $H^0(X,L)$ as 
$H^0(X,L)=\oplus_{i}V_i$ where the eigenvalue of the action on $V_i$ is $t^{w_i}$. 
Then, as explained in \cite[a few lines after Lemma 2]{Don05}, any other choice is given by $p(t)\circ f$ where $p\in (P(\lambda))_k$ 
and $p(t)$ is defined by $p(t)|_{{\rm End}(V_i,V_j)}=t^{w_{j}-w_{i}}p|_{{\rm End}(V_i,V_j)}$. 
Note that $\lambda_{p(t)\circ f}=p\lambda_{f}p^{-1}$. 

Thus, each test configuration is an equivalence class of one parameter subgroups 
divided by the relation ${\rm (ii)}$. And it is also easy to see that the relation 
${\rm (i)}$ (resp.\ ${\rm (iii)}$) corresponds to the base change, i.e. P-pre-equivalence 
(resp.\ R-equivalence) introduced in Definition \ref{eq.tc}. 
Hence, test degenerations with very ample polarization of exponent $1$ 
correspond to one parameter subgroups of 
${\rm GL}(H^0(X,L))$ modulo ${\rm (i), (ii), (iii)}$ so that they correspond to 
rational points of the building $|\Delta(G)|_{\mathbb{Q}}$. 

\subsection{Almost trivial test configurations}\label{patho.ss}

In \cite{LX11}, it is recently pointed out that 
in the definition of K-stability, we should allow some non-trivial but ``almost 
trivial" test configurations having GIT weights to vanish, 
or simply dismiss those 
test configurations of a specific type from our mind. 
Rigorously speaking, it is characterized by the following condition. 
Indeed, those should have vanishing Donaldson-Futaki invariants by 
the formula \cite{Wan08}, \cite{Od11}. 

\begin{Def}[{\cite[Definition 1]{Stp11}}]
A test configuration $(\mathcal{X},\mathcal{L})$ is said to be {\it almost trivial} 
\footnote{It is called {\it trivial in codimension $2$} in \cite{Stp11} but 
I use this terminology by the following reason. That is, as Stoppa taught me, 
the way of using phrase as ``$\cdots$ in codimension $2$" may be familiar to 
differential geometers but it is 
not common for algebro-geometers and may well be confusing for them. 
The meanings are completely the same. }
if $\mathcal{X}$ is $\mathbb{G}_m$-equivariantly isomorphic to the product 
test configuration, away from a closed subscheme of codimension at least $2$. 
\end{Def}

By taking the fact that they are not necessarily trivial (cf.\ \cite[Example 1]{LX11}) 
into account, K-stability is slightly reformulated as follows. 

\begin{Def}[{cf.\ \cite[subsection 2.2]{LX11}, \cite[Definition 2]{Stp11}}]\label{s.Kst}
A polarized variety $(X,L)$ is K-stable if ${\rm DF}(\mathcal{X},\mathcal{L})>0$ 
for all test configurations {\it which are not almost trivial}. 
\end{Def}

Now, we give a partial characterization of  ``almost triviality" from 
different points of view as follows. We note that the converse does not hold in 
general. 
%As a preparation, let us recall the following definition and propositions. 

%\begin{Def}[{\cite[Definition 3.7]{Od11}}]\label{pn}
%A semi test configuration $(\mathcal{X}, \mathcal{M})$ is \textit{partially normal} 
%if any prime divisor supported on the singular locus of $\mathcal{X}$ 
%projects surjectively onto $\mathbb{A}^{1}$. 
%\end{Def}

%\begin{Prop}[{\cite[Proposition 3.8]{Od11}}]\label{partial normalization}
%For an arbitrary test configuration $(\mathcal{X},\mathcal{M})$, 
%there exists a finite surjective birational morphism $f \colon \mathcal{Y}\rightarrow %\mathcal{X}$, where $(\mathcal{Y}, f^{*}\mathcal{M})$ is a partially normal test %configuration, with $\DF(\mathcal{Y}, f^{*}\mathcal{M})\leq\DF(\mathcal{X},%\mathcal{M})$. 
%\end{Prop}

%\begin{Prop}[{\cite[Proposition 3.10]{Od11}}]\label{toblup}

%For an arbitrary partially normal test configuration $(\mathcal{X},\mathcal{M})$, 
%there is a flag ideal $\mathcal{J}$ and $r,s\in \mathbb{Z}_{>0}$ such that its blow up 
%$(\mathcal{B}:=Bl_{\mathcal{J}}(X\times \mathbb{A}^{1}), \mathcal{L}^{\otimes r}(-E))$  %is a semi test configuration, which is Gorenstein in codimension $1$, dominating %$(\mathcal{X},\mathcal{M}^{\otimes s})$ by a morphism 
%$f\colon \mathcal{B}\rightarrow \mathcal{X}$ such that $\mathcal{L}^{\otimes %r}%(-%E)=f^{*}\mathcal{M}$ and $\DF(\mathcal{B},\mathcal{L}^{\otimes %r}%(-%E))=\DF(\mathcal{X},\mathcal{M}^{\otimes s})$.  
%\end{Prop}

%If we take a point of view from one parameter subgroup, the pathological 
%configurations are characterized by the following 
%condition. 

\begin{Prop}\label{p.tc}
Assume $X$ is embedded into a projective space $\mathbb{P}(H^0(X,L^{\otimes r}))$ 
by the complete linear system $|L^{\otimes r}|$ with some $r\in \mathbb{Z}_{>0}$. 
Take a diagonalization of the $\mathbb{G}_m$-action $\lambda$ on $H^0(X,L^{\otimes r})$ by the eigenvectors $X_0,\cdots,X_c,X_{c+1},\cdots,X_N$ where $\mathbb{G}_m$ 
acts on $X_i$ by $t^{a_i}$. Assume that $a_0=\cdots=a_{c}<a_{c+1}\leq\cdots\leq a_N$. 

Under the assumptions above, if the associated test configuration is 
almost trivial, then $c>\dim (X)$ holds and the linear subspace $V(X_0,\cdots,X_c)$ of $\mathbb{P}(H^0(X,L^{\otimes r}))$, where $X_0,\cdots,X_c$ vanish, 
does not intersect with $X$, i.e., $V(X_0,\cdots,X_c)\cap X
=\emptyset$. 
\end{Prop}

\begin{proof}
%Recall that for any test configuration, 
%we can take a flag ideal $\mathcal{J}$ whose blow up semi test configuration 
%$({\rm Bl}_{\mathcal{J}}(X\times \mathbb{A}^1),\mathcal{L}^{\otimes r}(-E))$ 
%with the same Donaldson-Futaki invariant dominates 
%original family (cf.\ \cite{Od11}) and apply it to our case. In our setting, with the %terminology on diagonalization, recall that the flag ideal is defined as $\mathcal{J}=
%\sum I_{<X_0,\cdots,X_{b_i}>}t^i$ where $b_i:={\rm max}\{ j \mid a_{j}\leq i\}$ and 
%$I_{<X_0,\cdots,X_{b_i}>}$ is the coherent ideal corresponding to the linear subspace %defined by ${<X_0,\cdots,X_{b_i}>}$. Thus if $V(X_0,\cdots,X_c)\cap X =\emptyset$, 
%the corresponding test configuration $\mathcal{X}_\lambda$ has same 
%Donaldson-Futaki 
%invariant as the trivial test configuration has, i.e., zero, and hence 
%$\mathcal{X}_\lambda$ and ${\rm Bl}_{\mathcal{J}}(X\times \mathbb{A}^1)$ 
%are isomorphic in codimension $1$. Hence, $\mathcal{X}_\lambda$ is almost trivial. 

%On the other hand, l
Let us assume $V(X_0,\cdots,X_c)\cap X
\neq \emptyset$ and prove that the associated test configuration $\mathcal{X}_\lambda$ is not almost trivial. 
Take a closed point $x$ of $V(X_0,\cdots,X_c)\cap X$, then linear projection 
from $x$ is not birational, i.e. degree should be at least $2$ as $x\in X$. 
Thus, the resulting test configuration should not have a central fiber which is birational 
to original $X$. 

On the other hand, if $c\le \dim(X)$, the linear projection from $V(X_0,\cdots,X_c)$ 
to the corresponding projective space is surjective. If it is generically finite, $c=\dim(X)$ follows from it by comparing dimensions. However, in that case, it should not be birational because $X$ is non-degenerate in $\mathbb{P}(H^0(X,L^{\otimes r}))$. This completes the proof. 
\end{proof}

In other words, if a test configuration $(\mathcal{X},\mathcal{M})$ is almost trivial then 
its associated flag ideal $\mathcal{J}\subset \mathcal{O}_{X\times \mathbb{A}^1}$ 
(cf.\ \cite{Od11}) is of the form $(t^N)$ and the natural dominating morphism from  $\mathcal{B}:=Bl_{\mathcal{J}}(X\times \mathbb{A}^1)(\cong (X\times \mathbb{A}^1))$ to $\mathcal{X}$ is isomorphic 
away from a closed subscheme of $\mathcal{X}$ with codimension at least $2$. 
Of course, this is a characterization of the almost triviality. 
Thus, given a test configuration which is not almost trivial, the associated flag ideal  $\mathcal{J}$ is not of the form $(t^N)$ or, if it is of such a form, $\mathcal{B}\rightarrow \mathcal{X}$ is not isomorphic in codimension $1$. For the latter case, we have $\DF(\mathcal{X},\mathcal{M})>0$ by the arguments of \cite[section 5]{RT07} or the formula 
for the Donaldson-Futaki invariant \cite{Od11}. 

On the other hand, once a semi test configuration of 
blow up type $(\mathcal{B}:=Bl_{\mathcal{J}}(X\times \mathbb{A}^1),\mathcal{L}^{\otimes r}(-E))$ is given, we can construct a contracted (ample) test 
configuration $(\Proj 
\oplus_{k\ge 0} H^0(X\times \mathbb{A}^1,\mathcal{J}^k p_1^* L^{\otimes rk}), \mathcal{O}(1))$ 
with obviously the same Donaldson-Futaki invariant. 

In sum, we have the following, which is a corrected proposition of \cite{Od11}. 

\begin{Cor}[{cf.\ \cite[Corollary 3.11]{Od11}}]\label{usefulness}
Assume $X$ is Gorenstein in codimension $1$. 
Then, a polarized variety $(X,L)$ is K-stable (\underline{in the sense of Definition \ref{s.Kst}}) 
if and only if for all semi test configurations of the blow up type in \cite{Od11} 
$($i.e.,  $(\mathcal{B}=Bl_{\mathcal{J}}(X\times \mathbb{A}^{1}), \mathcal{L}^{\otimes{r}}(-E))$ with $r\in \mathbb{Z}_{>0}$, flag ideal $\mathcal{J}$ \underline{which is not of the form $(t^N)$} and $\mathcal{B}$ Gorenstein in codimension $1$ $)$, the Donaldson-Futaki invariant is positive. 
\end{Cor}

\noindent
Originally, the author claimed the corresponding statement for original K-stability 
\cite{Don02} which does not work since he ignored the case $\mathcal{J}=(t^N)$ so that 
the blow up morphism $\mathcal{B}\rightarrow X\times \mathbb{A}^1$ is trivial. 
The proof works for proving the above statement. I apologize for this inaccuracy. 

Thus, with this slightly modified K-stability notion and Corollary \ref{usefulness}, 
the results of \cite{Od11} the sequels are justified in the naturally corrected form, 
or partially work without change. 

We explain a locus of the spherical building, inside which these pathologies occur. 

\begin{Def}
Assume $T$ is a maximal torus of $G={\rm GL}(H^0(X,L^{\otimes r}))$ 
and take the corresponding diagonalizing basis $X_0,\cdots,X_N$ of 
$H^0(X,L^{\otimes r})$. In the rational sphere $S(T,\mathbb{Q}):=
({\rm Hom}_{\rm alg.grp}(\mathbb{G}_m,T)\otimes \mathbb{Q})/\mathbb{Q}_{>0}$, 
we define a subset 
$$
S'(T,\mathbb{Q}):=
\{(a_i) \mid V(\{ X_j \mid a_j= {\rm min}_i \{a_i \}\}) \cap X=\emptyset \}. 
$$

\noindent
Note that $S'(T,\mathbb{Q})$ has dense complement 
as for each $(a_i)\in S'(T,\mathbb{Q})$ we have 
$\# \{ j \mid a_j= {\rm min}_i \{a_i \}\} >1$. It is because the corresponding linear subspace should not be of codimension $1$ as $X\subset \mathbb{P}(H^0(X,L^{\otimes r}))$ is non-degenerate. 

\end{Def}

\begin{Cor}[{of Proposition \ref{p.tc}}]
In the geometric realization of the spherical building $\Delta(G)$, 
almost trivial test degenerations are parametrized by $\phi(\cup_{T} S'(T,\mathbb{Q}))$ whose complement $($in $|\Delta(G)|_{\mathbb{Q}}$$)$ is dense. 
\end{Cor}

\section{Weight function}\label{s.wt}

We give the definitions of normalized Chow weights and optimal destabilizing test degeneration as follows. 
See \cite[Chapter 1, 2]{Mum65} to review basics of GIT weights. 
Assume $G:={\rm SL}(H^0(X,L^{\otimes r}))$ in this section and 
fix a maximal torus $T$ of $G$. 
With a Weyl group invariant norm $|\cdot|$ on 
${\rm Hom}_{\rm alg.grp}(\mathbb{G}_m,T)\otimes_{\mathbb{Z}} \mathbb{R}$, 
denote the normalized GIT weight function by $\nu_T:=\frac{\mu_T}{|\cdot|}$, 
where $\mu_T$ is the usual GIT weight function restricted to the set of 
one parameter subgroups of $T$.  
Recall that $\nu_T\colon |\Delta(T)|_{\mathbb{Q}}\rightarrow \mathbb{R}$ can be naturally extended to $\nu$ on whole $|\Delta(G)|_{\mathbb{Q}}$ by taking conjugate which sit in $T$ for each one parameter subgroup. Note that 
$\Delta(G)=\Delta({\rm SL}(H^0(X,L^{\otimes r})))$ can be and will be 
naturally identified with 
$\Delta({\rm GL}(H^0(X,L^{\otimes r})))$ by their definitions. 

\begin{Def}
For a given embedded projective variety $[X\subset \mathbb{P}]$ and 
one parameter subgroup $\lambda\colon \mathbb{G}_m\rightarrow 
{\rm SL}(H^0(\mathbb{P},\mathcal{O}(1)))$, its {\it normalized Chow weight} 
means $\nu(\bar{\lambda};{\rm Chow}(X\subset \mathbb{P}))$, 
where ${\rm Chow}(X\subset \mathbb{P})$ means the associated Chow point. 
\end{Def}

One of the main issue of this notes is the existence of the following object. 

\begin{Def}
Assume $[X\subset \mathbb{P}]$ is Chow unstable. 
Then, its {\it maximally destabilizing test degeneration} means 
the test degeneration which corresponds to a point of $|\Delta(G)|_{\mathbb{Q}}$ 
with representing one parameter subgroup in ${\rm SL}(H^0(X,L^{\otimes r}))$ 
which have the minimum (so negative) normalized Chow weight. 
\end{Def}

\noindent
Note that Corollary \ref{C.DO} which states its existence and uniqueness simply 
follows from \cite{Rou78}, \cite{Kem78} combined with Theorem \ref{tc.bld} 
of this notes. 

Now, we briefly explain the continuity property of normalized Chow weights (Theorem \ref{Cwt}). 

\begin{proof}[Proof of Theorem \ref{Cwt}]

Recall the following Lemma due to Mumford. 

\begin{Lem}[{cf.\ \cite[Chapter 2, Proposition 2.14]{Mum65}}]
With finite linear functions on ${\rm Hom}_{{\rm alg.grp}}(\mathbb{G}_m,T)\otimes_{\mathbb{Z}} \mathbb{R}$ with rational coefficients $l_1,\cdots,l_m$, 
and a map $I\colon G\rightarrow \{$ non-empty subsets of $\{1,\cdots,m\}\}$, 
the weight function $\mu$ is of the form: 
$$
\mu_{gTg^{-1}}(\lambda)={\rm max}\{l_i(\bar{\lambda}) \mid i\in I(g)\}. 
$$
\noindent
Here, $I$ satisfies ``lower semicontinuity" in the sense that for any $J\subset 
\{1,\cdots,m\}$, $\{g\in G\mid J\subset I(g)\}$ is open. 
\end{Lem}

\noindent
Although the last sentence is not explicitly written in Mumford's exposition, 
it is essentially proved there. See also \cite{Kem78}. (He used his own terminology 
``states" which corresponds to our map $I$. )
The continuity of the right hand side implies continuity of normalized 
Chow weights $\nu$ on $|\Delta(G)|_{\mathbb{Q}}$ 
and moreover, the last sentence implies 
lower semi-continuity with respect to the stuck topology. 

\end{proof}

\begin{ack}
The author is grateful to Professor Shigefumi Mori for his helpful advice. 
He also would like to thank Chi Li, Yoshiki Oshima, 
Jacopo Stoppa, Song Sun and Gabor Szekelyhidi 
for their helpful comments. 
The author is partially supported by the Grant-in-Aid for Scientific Research (KAKENHI No.\ 21-3748) and the Grant-in-Aid for JSPS fellows (PD). 
\end{ack}

\end{document}